\documentclass{amsart}
\usepackage{graphicx}
\usepackage{amssymb}
\usepackage{amsfonts}
\usepackage{hyperref}
\usepackage{mathrsfs}
  [1995/08/01 v0.1 Double stroke roman fonts]

\def\ds@whichfont{dsrom}
\relax

\DeclareMathAlphabet{\mathds}{U}{\ds@whichfont}{m}{n}

\swapnumbers
\newtheorem{theorem}{Theorem}[section]
\newtheorem{lemma}[theorem]{Lemma}
\newtheorem{corollary}[theorem]{Corollary}
\newtheorem{proposition}[theorem]{Proposition}
\theoremstyle{definition}

\newtheorem{assumption}[theorem]{Assumption}
\newtheorem{remark}[theorem]{Remark}
\newtheorem{example}[theorem]{Example}

\numberwithin{equation}{section}
\theoremstyle{plain}

\numberwithin{equation}{section} 
\numberwithin{figure}{section} 
\theoremstyle{plain}
\theoremstyle{plain}
\theoremstyle{remark}
\newtheorem*{acknowledgement*}{Acknowledgement}

\newcommand{\cD}{{\mathcal D}}
\newcommand{\cE}{{\mathcal E}}
\newcommand{\cF}{{\mathcal F}}
\newcommand{\cG}{{\mathcal G}}
\newcommand{\cH}{{\mathcal H}}

\newcommand{\cL}{{\mathcal L}}
\newcommand{\cM}{{\mathcal M}}

\newcommand{\cS}{{\mathcal S}}

\newcommand{\cW}{{\mathcal W}}
\newcommand{\cX}{{\mathcal X}}
\newcommand{\cY}{{\mathcal Y}}

\newcommand{\te}{{\theta}}

\newcommand{\Om}{{\Omega}}
\newcommand{\om}{{\omega}}
\newcommand{\ve}{{\varepsilon}}
\newcommand{\del}{{\delta}}
\newcommand{\Del}{{\Delta}}
\newcommand{\gam}{{\gamma}}

\newcommand{\sig}{{\sigma}}
\newcommand{\al}{{\alpha}}
\newcommand{\be}{{\beta}}

\newcommand{\bbC}{{\mathbb C}}
\newcommand{\bbE}{{\mathbb E}}

\newcommand{\bbN}{{\mathbb N}}

\newcommand{\bbR}{{\mathbb R}}

\newcommand{\bbI}{{\mathbb I}}

\newcommand{\bbW}{{\mathbb W}}
\newcommand{\bbS}{{\mathbb S}}

\begin{document}
\title[]{On Eagleson's theorem in the non-stationary setup}
 \author{Yeor Hafouta \\
\vskip 0.1cm
Department  of Mathematics\\
The Ohio State University}
\email{yeor.hafouta@mail.huji.ac.il, hafuta.1@osu.edu}

\maketitle
\markboth{Y. Hafouta}{Eagleson's theorem}
\renewcommand{\theequation}{\arabic{section}.\arabic{equation}}
\pagenumbering{arabic}

\begin{abstract} 
A classical result due to Eagleson states (in particular) that if appropriately normalized Birkhoff sums generated by a measurable function and an ergodic probability preserving transformation converge in distribution, then they also converge in distribution with respect to any probability measure which is 
absolutely continuous with respect to the invariant one. In this  note we prove  several quantitative and infinite-dimensional versions of Eagleson's theorem for some classes of non-stationary stochastic processes which satisfy certain type of decay of correlations.
\end{abstract}

\section{introduction}
Let $(\Om,\cF,\mu,T)$ be an probability preserving system (p.p.s.) and let $f:\Om\to\bbR^d$ be a measurable function. Then the partial sums $S_nf=S_nf(x)=\sum_{n=0}^{n-1}f(T^nx)$ are random variables, where $T^n=T\circ T\cdots\circ T$  and $x$ is chosen at random according to $\mu$ (i.e. the probability that $x$ belongs to a measurable set $A$ is $\mu(A)$). Note that any  discrete time vector-valued stationary process $Y_0,Y_1,...$ has the form $Y_n=f(T^nx)$ for some p.p.s and a measurable function $f$. Given such a function $f$, an important question in probablity and ergodic theory is whether  $(S_nf-a_n) f/b_n$ converges
in distribution as $n\to\infty$, for some sequences $(a_n)$ and $(b_n)$ so that $\lim_{n\to\infty}b_n=\infty$. A related and more general question is whether the continuous time processes $W_n(t)=(S_{[nt]}f-a_{[nt]})/b_n$ converge in distribution.

In certain circumstances there is a given reference measure (e.g. the Lebesgue measure) which is absolutely continuous with respect to the invariant measure $\mu$. 
A classical result due to Eagleson \cite{Eag} insures (for $d=1$) that the weak convergence of $(S_n-a_n)/b_n$ with respect to $\mu$ is equivalent to the weak convergence with respect to the reference measure (with the same limits), which is  more natural since usually the density of $\mu$ does not have an explicit form. Since then this result was extended to more general processes which are not necessarily real-valued, see for instance \cite{ZW}.
In particular, one can also consider  vector-valued functions $f$, as well as random continuous time processes of the form $W_n(t)=(S_{[nt]}f-a_{[nt]})/b_n$, which gives us a version of Eagleson's theorem in the context of the the weak invariance principle (WIP), which states that the stochastic process $\cW_n(t)=n^{-1/2}S_{[nt]}(f-\mu(f))$ converge in distribution as $n\to\infty$ in the Skorokhod space $D([0,\infty),\bbR^d)$ towards a Gassian process (other normaliztions $b_n$ can be considered). 
 Another application of \cite{ZW} is to the, so called, iterated WIP, which yields a certain type of smooth approximations of stochastic differential equations for suspension flows built over non-uniformly expanding or hyperbolic maps \cite[Theorems 2.1, 2.2]{KM}.
 We also refer to \cite{Go} for additional results which also have applications to the strong invariance principle. 

The results described in the latter paragraph concern partial sums $S_n=\sum_{j=0}^{n-1}Y_n$ generated by vector-valued stationary process $\{Y_n\}$ defined on a probability space $(\Om,\cF,\mu)$,
and in this paper we prove certain versions of Eagleson's theorem for non-stationary sequences of vector-valued processes $\{Y_n\}$ satisfying certain mixing (decay of correlations) conditions which hold true for many sequential dynamical systems including the ones arising as realizations of random dynamical systems, as well as for wide classes of inhomogeneous Markov chains and other mixing sequences. Note that for real-valued $Y_n$'s 
Eagleson's results \cite{Eag} also apply when the tail-$\sigma$ algebra of $\{Y_n\}$ is trivial, which will be the case in most of the examples we have in mind. 
We start with the above vector-valued case, but when $Y_n$'s are real-valued we also prove a quantitative version,  which means that we obtain explicit estimates on the convergence rate in the weak convergence with respect to a measure $m$ which is absolutely continuous with respect to $\mu$, in terms of the rate  in the corresponding convergence with respect to the original measure $\mu$ (for which the mixing conditions originally hold). The question of optimal convergence rate will also be addressed, as well as the problem of re-centering and re-normalizing after changing the measure. 
We also prove a non-stationary  version of Eagleson's theorem for continuous time stochastic processes of the form $\cS_n(t)=S_{[nt]}f/b_n$ (i.e. a version for the WIP).
Our results, for instance yield the WIP for the compositions of random Anosov or expanding maps considered in \cite{DFGTV0} and \cite{DFGTV2}, with respect to the Lebesgure measure and not only with respect to the random equivariant measures (the WIP for such maps follows from the almost sure invariance principles which were obtained in \cite{DFGTV0} and \cite{DH}).
Finally, we will also discuss a version of Eagleson's theorem in the, so called, iterated WIP, which we expect to  have applications in smooth approximations of stochastic differential equations for non-stationary suspension flows built over random and sequential dynamical systems (i.e. in a non-stationary version of \cite{KM}).

\section{Preliminaries and examples}\label{Sec1}
Let $(\cE,\cF,\mu)$ be a probability space,  $\cX_0,\cX_1,\cX_2,...$ be measurable spaces and  $X_0,X_1,X_2,...$ be a sequence of measurable functions on $\cE$, so that $X_i$ takes values in $\cX_i$ for each $i$. In this paper we are interested in sequences so that the  partial sums $S_n g=\sum_{j=0}^{n-1}g_j(X_j,X_{j+1},...)$ satisfy the central limit theorem for large classes of sequences of functions $\{g_j\}$, namely there are sequences $(a_n)$ and $(b_n)$ which depend on $\{g_j\}$ so that $(b_n)$ tends to $\infty$ and $(S_n g-a_n)/b_n$ converges in distribution towards the standard normal law.
In general, for the CLT to hold true for a large class of sequences $\{g_j\}$ a certain type of asymptotic independence between $\{X_0,...,X_n\}$ and $\{X_{n+k},X_{n+k+1},....\}$ as $k\to\infty$ is required.
As mentioned in the abstract, our standing assumption is a certain type of decay of correlations, which is a quantitative way of  measuring such dependence.
\begin{assumption}\label{Ass1}
There exists a sequence $(\del_n)$ which converges to $0$ as $n\to\infty$ and a set $B$ of real integrable functions on $\cE$ equipped with a ``norm" $\|\cdot\|$ so that for all $n$, a function $s\in B$ and a bounded complex-valued function $f=f(x_{n},x_{n+1},x_{n+2},...)$ we have
\begin{equation}\label{Dec}
\left|\int s(x)f(\overline{X}_n(x))d\mu(x)-\int s(x)d\mu(x)\cdot \int f(\overline{X}_n(x))d\mu(x)\right|\leq \|s\|\|f\circ \overline{X}_n\|_\infty\del_n
\end{equation}  
where $\overline{X}_n(x)=(X_{n}(x),X_{n+1}(x),X_{n+2}(x),...)$ and $\|f\circ \overline{X}_n\|_\infty$ is the essential supremum of the function $f(\overline{X}_n(x))$ with respect to $\mu$.
\end{assumption}
This assumption holds true in a variety of models, which will be described in Examples \ref{Eg1} and \ref{Eg2}.
Let us now explain why we only need $\|f\circ \overline{X}_n\|_\infty$ to appear on the right hand side of \eqref{Dec} (and not a smaller norm). Let $g_j, j\geq 0$ be a sequence of functions on $\cX_{j}\times\cX_{j+1}\times...$ and set $S_ng(x)=\sum_{j=0}^{n-1}g_j(\overline{X}_j(x))$. Then the goal in this paper is to investigate the limit (distributional) behavior of  $S_n g$ (and related infinite dimensional processes) when $x$ is distribution according to measures $\nu$ which are absolutely continuous with respect to $\mu$ and $r=d\nu/d\mu$ belongs to the $L^p$-closure of $B$ for some $p\geq1$.
The idea behind the the proofs  is that for a density $r\in B$ and any real $t$, integers $0\leq k<n$ and a normalizing sequence $(b_n)$ so that $\lim_{n\to\infty} b_{n}=\infty$ we have
\[
\int r(x) e^{itS_ng(x)/b_n}d\mu(x)=\int r(x) e^{it(S_ng(x)-S_kg(x))/b_n}d\mu(x)+O(t/b_n)A_k
\]
where   $A_k=\int r(x)|S_k g(x)|d\mu(x)$.
 Now, since $e^{it(S_ng(x)-S_kg(x))/b_n}$ is a bounded function of  $\overline{X}_k$, using \eqref{Dec} we get that 
$$
\int r(x) e^{itS_ng(x)/b_n}d\mu(x)=\int e^{itS_ng(x)/b_n}d\mu(x)+O(t/b_n)B_k+O(\delta_k)
$$
where $B_k=\int (r(x)+1)|S_k g(x)|d\mu(x)$.
By choosing $k=k_n$ appropriately so that $\lim_{n\to\infty}k_n\to\infty$ (and other restrictions hold, depending on the result we want to prove) we see that the characteristic function of $S_n g/b_n$ with respect to $\nu=rd\mu$ can be controlled by the corresponding one with respect to $\mu$ on appropriate domains.

Before formulating our main results let us discuss two main types of examples which satisfy Assumption \ref{Ass1}.
\begin{example}[Random and sequential dynamical systems]\label{Eg1}
Let $\cX_0=\cE$ and $T_0,T_1,T_2,...$ be a sequence of maps so that $T_j:\cX_j\to\cX_{j+1}$. Let $X_0(x)=x$ and  set $X_j(x)=X_0(T_0^jx)=T_0^jx$, where $T_n^m=T_{n+m-1}\circ\cdots\circ T_{n+1}\circ T_{n}$ for all $n$ and $m$. Then $X_{j+n}=X_j\circ T_j^n$ for every $n$ and $j$. Therefore, $\overline{X}_n(x)$ depends only on $X_n(x)$ and (\ref{Dec}) becomes  
\begin{equation*}
\left|\int s(x)f(T^nx)d\mu(x)-\int s(x)d\mu(x)\cdot \int f(T^n x)d\mu(x)\right|\leq \|s\|\|f\circ T^n\|_{L^\infty(\mu)}\del_n.
\end{equation*}
This condition (with an appropriate $\mu$) is satisfied for appropriate $B$'s and norms $\|\cdot\|$ for many sequential and random dynamical systems, where in many of the examples we can even replace $\|f\circ T^n\|_{L^\infty(\mu)}$ with the corresponding $L^1(\mu)$-norm. We refer the readers to
 \cite{ABR}, \cite{B1}, \cite{Buzzi}, \cite{DFGTV1}, \cite{DFGTV2}, \cite{HafSDS}, \cite{HK}
and \cite{Kif LimThms}, \cite{MSU} and references therein. We note that in some of these papers the case when  $T_j=T_{\te^j\om}$ is a random stationary family of maps is considered, where $(\Om,\cF,P,\te)$ is a measure preserving system, and  $T_\om,\,\om\in\Om$ is a measurable in $\om$ family of maps. We note that in most of the above papers $B$ is a normed space which is dense in $L^p(\mu)$ for every finite $p\geq1$. 

Remark that in \cite{ABR} the authors obtained almost sure rates of mixing for certain classes random hyperbolic maps $T_\om$. The authors of \cite{ABR} show that these  maps admit a random tower extension $(\Del_\om, F_\om)$, first introduced in \cite{BBM} (which generalizes \cite{Y2} to the random case). The random tower inherits the random hyperbolic structure from the original maps $T_\om$, and after collapsing stable manifolds, the statistical properties of the original maps (with respect to the random physical measure) are reduces to the resulting ``projected" random tower $(\bar \Del_\om, \bar F_\om)$, see \cite[Section 2.3]{ABR}. A direct application of  \cite[Theorem 2.5]{ABR} shows that
Assumption \ref{Ass1} holds true (for $P$-a.a. $\om$) on the projected tower with $B=B_\om$ being  space of H\"older continuous functions on $\Del_\om$ and $\mu=\mu_\om$, where $\mu_\om$ is the absolutely continuous equivariant measure (i.e. $(\bar F_\om)_*\mu_\om=\mu_{\te \om}$).
\end{example}

\begin{example}[Non-stationary mixing stochastic processes]\label{Eg2}
Let $X=\{X_j\}$ be a sequence of random variables defined on the same probability space $(\cE,\cF,\mu)$. 
For each $n\leq m$ we denote by $\cF_{n,m}$ the $\sig$-algebra generated by the random variables $X_{n},X_{n+1},...,X_{n+m}$.  Let $\cF_{n,\infty}$ denote the $\sig$-algebra generated by the random variables $X_{j},\,j\geq n$. For any two sub-$\sig$-algebras $\cG$ and $\cH$ of $\cF$, set
\[
\al(\cG,\cH)=\sup\left\{|P(A\cap B)-P(A)P(B)|:\,A\in\cG,\,B\in\cH\right\}.
\]
The coefficient $\al(\cG,\cH)$ measures the dependence between $\cG$ and $\cH$ and it is one of the classical mixing coefficients used in the literature (often referred to as the strong mixing coefficient). For each $n\geq1$ set 
\[
\al_n=\al_n(X)=\sup_{k\geq0}\al(\cF_{0,k},\cF_{k+n,\infty}).
\]
The sequence $\{X_n\}$ is called $\al$-mixing if $\lim_{n\to\infty}\al_n=0$.
A concrete example for non-stationary $\al$-mixing processes
are the inhomogeneous Markov chains considered in \cite{VarSeth}. See \cite{DS}, \cite{Douk1}, \cite{Douk2}, \cite{Tru1} and \cite{Tru2} for other examples of $\al$-mixing non-stationary processes.

Let $(\gam_n)$ be a sequence which converges to $0$ as $n\to\infty$ and let $B$ be the set of all functions $s$ on $\cE$ so that for every sufficiently large $n$,
\[
\beta_{1,n}(s):=\|s-\bbE[s|X_0,...,X_{n}]\|_{L^1(\mu)}\leq \gam_n.
\]  
It is clear that $B$ contains all the random variables of the form $s=s(X_0,...,X_n)$.
Let $p\geq 1$ be  finite and $\|\cdot\|$ be the $L^p(\mu)$-norm. Then for all $s\in B$ so that $\int sd\mu=1$ we have  
\begin{eqnarray*}
\left|\int s(x)f(\overline{X}_n(x))d\mu(x)-\int s(x)d\mu(x)\cdot \int f(\overline{X}_n(x))d\mu(x)\right|\\\leq \left|\int s_{[\frac n2]}(x)f(\overline{X}_n(x))d\mu(x)-\int s_{[\frac n2]}(x)d\mu(x)\cdot \int f(\overline{X}_n(x))d\mu(x)\right|\\+2\|f\circ\overline{X}_n\|_{L^\infty}\|s-s_{[\frac{n}2]}\|_{L^1}
\end{eqnarray*}
where $s_n=\bbE[s|X_0,...,X_{n}]$. Since $\int sd\mu=1$ we have
\[
\|s-s_{\frac{n}2}\|_{L^1}\leq \gam_{[\frac n2]}\leq \|s\|_{L^p}\gam_{[\frac n2]}.
\]
Next, by Corollaries A.1 and A.2 in \cite{HallHyde} and since conditional expectations contract $L^p$-norms, for every $p\geq 1$ we have
\begin{eqnarray*}
\left|\int s_{[\frac n2]}(x)f(\overline{X}_n(x))d\mu(x)-\int s_{[\frac n2]}(x)d\mu(x)\cdot \int f(\overline{X}_n(x))d\mu(x)\right|\\\leq 6\left(\al(\cF_{0,[\frac n2]},\cF_{n,\infty})\right)^{1-\frac 1p}\|s\|_{L^p}\|f(\overline{X}_n)\|_{L^\infty}
\end{eqnarray*}
where we use the convention $\frac 1{\infty}:=0$. We conclude that
in the above circumstances the conditions in Assumption \ref{Ass1} hold true with 
 $\|s\|=\|s\|_{L^p}$ and 
 $\del_n=6\big(\al_{[\frac n2]}\big)^{1-\frac 1p}+2\gam_{[\frac n2]}$.
\end{example}

\section{Vector-valued processes}
Henceforth, when it is more convenient we will denote the integral of a function $f$ with respect to $\mu$ by $\mu(f)$. We will also denote by $\mu_j$ the distribution of $X_j$. For each $n$ set $\cY_n=\cX_n\times\cX_{n+1}\times...$.
Let $d\geq1$, $g_j:\cY_j \to\bbR^d$ be a sequence of functions and $r:\cE\to\bbR$ be a non-negative function so that  $\int r(x)d\mu(x)=1$ (i.e. $r$ a probability density with respect to $\mu$). 
Consider the functions $S_n:\cE\to\bbR^d$ given by
\[
S_n(x)=\sum_{j=0}^{n-1}g_j(X_j(x),X_{j+1}(x),...)=\sum_{j=0}^{n-1}g_j(\overline{X}_j(x)).
\]
Let $\nu$ be the probability measure on $\cE$ defined by $d\nu=rd\mu$. We can view $S_n=S_n(x)$ as a random variable when $x$ is distributed according to either $\mu$ or $\nu$. We denote these random variables by $S_{n,\mu}$ and $S_{n,\nu}$, respectively. Our first result is the following:

\begin{theorem}\label{Theorem1}
Suppose that Assumption \ref{Ass1} holds true.
Assume also that that for some two conjugate exponents $p$ and $q$ we have that 
 $r$ lies in the $L^p(\mu)$-closure of $B\cap L^p(\mu)$ and $g_j\circ \overline{X}_j\in L^q(\mu)$ for all $j\geq0$.
 Then under Assumption \ref{Ass1}, for every sequence $(b_n)_n$  of positive numbers which tends to $\infty$ (as $n\to\infty$), the sequence $S_{n,\mu}/b_n$ converges in distribution  if and only if $S_{n,\nu}/b_n$ converges in distribution, and in the latter case both converge towards to the same limit.
\end{theorem}

For real valued $g_j$'s, Theorem \ref{Theorem1} follows from \cite{Eag} 
when the tail $\sig$-algebra $\mathcal T$ of the sequence $Y_j=g_j(X_j,X_{j+1},...)$ is trivial. Under Assumption \ref{Ass1}, it is clear that $\mathcal T$ is trivial when the $L^1$-closure of $B$ contains all integrable $\mathcal T$-measurable functions. Therefore, we essentially do not consider Theorem \ref{Theorem1} as a new result, but we still present a proof since later on we will adapt its arguments to obtain a quantitative version, as well as a  version corresponding to the weak invariance principle.

\begin{proof}
First, for any  real $t$ we  set $t_n=t/b_n$. By the Levi continuity theorem it is enough to show that for any  fixed $t$ we have
\[
\lim_{n\to\infty}|\mu(r\cdot e^{it_n S_n})-\mu(e^{it_n S_n})|=0.
\]
In the case when $r$ does not lie in $B$, given $\ve>0$ we can first
approximate $r$ within $\ve$ in $L^p(\mu)$ by a (nonnegative) $s\in B$ so that $\mu(s)=1$. Then for any real $t$ we have 
\[
|\mu(r\cdot e^{it_n S_n})-\mu(s\cdot e^{it_n S_n})|\leq \|r-s\|_{L^p(\mu)}<\ve.
\]
Therefore, it is enough to prove the theorem when $r\in B$ and the integrals $\int |S_n(x)|d\mu(x)$ and $\int r(x) |S_n(x)|d\mu(x)$ are finite for all natural $n$.

Next, since $\lim_{n\to\infty}b_n=\infty$, for any sequence $(c_k)$ the exists a (weakly increasing) sequence $(a_n)$ of natural numbers which tends to $\infty$ as $n\to\infty$ so that 
$c_{a_n}=o(b_n)$. It is also clear that we can assume that $a_n<n$.
Consider the sequence
\[
c_k=\int |S_k(x)|d\mu(x)+\int r(x) |S_k(x)|d\mu(x)
\]
and let $a_n$ be so that $c_{a_n}=o(b_n)$.

Next, by the mean value theorem,
\[
\left|\mu(r\cdot e^{it_n S_n})-\mu(r\cdot e^{it_n (S_n-S_{a_n})})\right|\leq |t|b_{n}^{-1}\mu(r\cdot |S_{a_n}|)\leq |t|c_{a_n}/b_n\to 0\text { as }n\to\infty.
\]
Relying on Assumption \ref{Ass1}, taking into account that $\mu(r)=1$ and that $S_{n}-S_{a_n}$ is a function of $\overline{X}_{a_n}$,
we have
\[
|\mu(r\cdot e^{it_n(S_n-S_{a_n})})-
\mu(e^{it_n(S_n-S_{a_n})})|\leq \|r\|\del_{a_n}\to 0\text { as }n\to\infty.
\] 
Finally, by the mean value theorem,
\[
\left|\mu(e^{it_n(S_n-S_{a_n})})-\mu(e^{it_n S_n})\right|\leq |t|b_{n}^{-1}\mu(|S_{a_n}|)\leq |t|c_{a_n}/b_n\to 0\text { as }n\to\infty.
\]
\end{proof}

\subsection{Recentering after change of measure}\label{SecCent}
In applications, it is often the case
where $(S_{n,\mu}-\bbE[S_{n,\mu}])/b_n$ converges in distribution, and this just means that we  replace $g_j$ with $g_j-\mu(g_j(\overline{X}_j))$ in the setup of the previous section. Applying Theorem \ref{Theorem1} we infer  that $(S_{n,\nu}-\bbE[S_{n,\mu}])/b_n$ converges in distribution, and to the same limit. The ``centering" term $\bbE[S_{n,\mu}]$ is not natural in the latter convergence, and it is natural to inquire whether $(S_{n,\nu}-\bbE[S_{n,\nu}])/b_n$ converges in distribution. When all the $g_j$'s are bounded and $r\in B$, under Assumption \ref{Ass1} we have
\[
\left|\bbE[S_{n,\nu}]-\bbE[S_{n,\mu}]\right|\leq \sum_{j=0}^{n-1}|\mu(r g_j(\overline{X}_j))-\mu(r)\mu(g_j(\overline{X}_j))|
\leq \|r\|\sum_{j=0}^{n-1}\del_j\|g_j\|_\infty.
\]
Therefore, if $\sum_{j=0}^{n-1}\del_j\|g_j\|_\infty=o(b_n)$ we obtain that the difference between
$(S_{n,\nu}-\bbE[S_{n,\nu}])/b_n$ and $(S_{n,\nu}-\bbE[S_{n,\mu}])/b_n$ converges almost surely to $0$, which yields the desired convergence in distribution of $(S_{n,\nu}-\bbE[S_{n,\nu}])/b_n$. Of course, the assumption that $g_j$'s are bounded can be weakened. For any sequence of $M_j>0$ we have 
\begin{eqnarray*}
\sum_{j=0}^{n-1}|\mu(r g_j(\overline{X}_j))-\mu(r)\mu(g_j(\overline{X}_j))|\leq \|r\|\sum_{j=0}^{n-1}M_j\del_j\\+
\sum_{j=0}^{n-1}|\mu\big((r+1) g_j(\overline{X}_j)\bbI(|g_j(X_j)|\geq M_j)\big)|.
\end{eqnarray*}
By the H\"older and  the Markov inequalities, for any $p_1,p_2,p_3\geq1$ so that $\frac{1}{p_1}+\frac{1}{p_2}+\frac{1}{p_3}=1$ we have 
\[
|\mu\big((r+1) g_j(\overline{X}_j)\bbI(|g_j(\overline{X}_j)|\geq M_j)\big)|\leq \|r+1\|_{p_1}\|g(\overline{X}_j)\|_{p_2}\|g(\overline{X}_j)\|_{p_2}^{p_2/p_3}M_j^{-p_2/p_3}
\]
where $\|f\|_p:=\|f\|_{L^p(\mu)}$ for any $p$ and a vector-valued function $f$ on $\cE$. 
This yields the following simple result:
\begin{proposition}\label{Cprop}
Suppose Assumption \ref{Ass1} hold, that $r\in B$ and that there are  $p_1,p_2,p_3$  and a sequence $(M_j)$ of positive numbers so that $\frac{1}{p_1}+\frac{1}{p_2}+\frac{1}{p_3}=1$, $\|r\|_{p_1}<\infty$
and 
\[
\cM_n:=\sum_{j=0}^{n-1}\left(M_j\del_j+\|g_j(\overline{X}_j)\|_{p_2}^{1+p_2/p_3}M_{j}^{-p_2/p_3}\right)=o(b_n).
\]
Then $|\bbE[S_{n,\nu}]-\bbE[S_{n,\mu}]|\leq \cM_n=o(b_n)$
 and therefore  $(S_{n,\nu}-\bbE[S_{n,\mu}])/b_n$ converges in distribution if and only if $(S_{n,\nu}-\bbE[S_{n,\nu}])/b_n$ converge in distribution, and to the same limit.
\end{proposition}
For instance, when $\sup_j\|g(\overline{X}_j)\|_{p_2}<\infty$ and $p_2/p_3$  are larger than $1$ and $\del_j\leq Cj^{-2-\ve}$ then we can take $M_j=j$ and get that $\cM_n$ is bounded in $n$. 
When $b_n=O(n^a)$ for some $a$ we can get rid of the power $\ve$ in the upper bound of $\del_j$.
 Of course, many other, more-explicit, moment conditions and decay rates of $\del_j$ can be imposed to insure that $\cM_n=o(b_n)$. We note that in most of the applications in Example \ref{Eg1} the space $B$ is composed of bounded functions, and so in this case we can take $p_1=\infty$.

\subsection{Renormaliztion after change of measure}\label{VarSec}
We assume here that the functions $g_j$ are real-valued.
For random dynamical systems the CLT holds true for the corresponding centered random birkhoff sums normalized by $b_n=\sqrt n$, but for sequential dynamical systems, inhomogeneous Markov chains and other non-stationary mixing sequences  usually the CLT holds true (only) with the self normalizing sequence $b_{n,\mu}=\sqrt{\text{Var}(S_{n,\mu})}$, especially because $\text{Var}(S_{n,\mu})$ may have various asymptotic behaviors.
 Of course, this requires us to assume that $\lim_{n\to\infty}{b_{n,\mu}}=\infty$, and we refer the readers to \cite{HafSDS} and \cite{DS} (when $g_j$ depends only on $(X_j,X_{j+1})$) for  characterizations of the latter convergence for several classes of sequential dynamical systems and  inhomogeneous Markov chains.
In Section \ref{SecCent} we showed that, under certain conditions,  the weak convergence of $(S_{n,\nu}-\bbE[S_{n,\nu}])/b_{n,\mu}$ follows from the convergence corresponding to $\mu$, and a natural question   is whether the convergence of $(S_{n,\nu}-\bbE[S_{n,\nu}])/b_{n,\nu}$ can also be derived, where $b_{n,\nu}=\sqrt{\text{Var}(S_{n,\nu})}$.

\begin{proposition}\label{Vprop}
Under Assumption \ref{Ass1},
 set $G_k=g_k\circ \overline{X}_k$, and assume that $\bbE_\mu[G_k]=0$ for all $k$.
Suppose also that $\del_k=C_1\del^k$ for some $C_1>0$ and $\del\in(0,1)$,  and that 
\begin{equation}\label{ExtCond}
\max(\|G_k\|,\|r\cdot G_k\|)\leq b\del^{-ak}\text{ for some } a,b>0 \text { and all } k\geq0.
\end{equation}
Moreover, assume that there are $c_0,C_0>0$, $\be\in(0,1)$ and $p>3$ so that for any $k\geq0$
\begin{equation}\label{p ass}
\|G_k\|_p\leq C_0e^{c_0 k^\beta}
\end{equation}
and $\|r\|_p<\infty$.
Then  there is $C>0$ so that for every $n\geq1$,
\[
\left|\bbE[S_{n,\mu}^2]-\bbE[S_{n,\nu}^2]\right|\leq C.
\]
Therefore,
\[
\left|\text{Var}(S_{n,\mu})-\text{Var}(S_{n,\nu})\right|\leq C+\cM_n^2
\]
where $\cM_n$ comes from Proposition \ref{Cprop}. Hence, when $\cM_n=o(b_{n,\mu})$ and $b_{n,\mu}$ tends to $\infty$ as $n\to\infty$ then  
\[
\lim_{n\to\infty}\frac{b_{n,\mu}}{b_{n,\nu}}=1
\]
and  therefore $(S_{n,\mu}-\bbE[S_{n,\mu}])/b_{n,\mu}$ converges in distribution if and only 
 $(S_{n,\nu}-\bbE[S_{n,\nu}])/b_{n,\nu}$ converges in distribution, and to the same limit.
\end{proposition}
In the circumstances of Example \ref{Eg1}, condition (\ref{ExtCond}) holds true  when $\|\cdot\|$ is an H\"older norm (as in \cite{MSU}, \cite{HK} or \cite{HafSDS}) or some total variation norm (as in \cite{DFGTV1}) and
 $g_j(\overline {X}_j(x))=h_j(T_0^j x)$, where $\|h_j\|$ are uniformly bounded in $j$. In the circumstances of Example \ref{Eg2}, the norm $\|\cdot\|$ is some $L^p$-norm and so (\ref{ExtCond}) will be satisfied if the functions $G_k$ are bounded $L^q$  for $q>p$ and $\|r\|_{q'}<\infty$, where $1/p=1/q+1/q'.$

\begin{proof}[Proof of Proposition \ref{Vprop}]
First, we have
\begin{equation}\label{Bas}
\left|\bbE[S_{n,\mu}^2]-\bbE[S_{n,\nu}^2]\right|\leq 2\sum_{0\leq k\leq j<n}\left|\text{Cov}_{\mu}(r, G_kG_j)\right|.
\end{equation}
Let $0\leq k\leq j<n$ be so that $j\geq (a+1)k$, where $a$ comes from \eqref{ExtCond}. Moreover let $\beta<\al<1$, where $\be$ comes from \eqref{p ass} and set
\[
\tilde G_j=G_j\bbI(|G_j|\leq e^{j^\al}).
\]
Let $q_2$ be the conjugate exponent of $p_2=p/2$ and $q_3$ be the conjugate exponent of $p_3=p/3$.
Let use write 
$$
\text{Cov}_{\mu}(r,G_kG_j)=\text{Cov}_{\mu}(r,G_k\tilde G_j)+\cD_{k,j}.
$$
Using (\ref{p ass}) and the H\"older and the Markov inequalities, we have
$$
|\cD_{k,j}|\leq \big(\|rG_kG_j\|_{p_3}^{1+p_3/q_3}+\|G_k G_j\|_{p_2}^{1+p_2/q_2}\big)e^{-c j^\al}\\\leq C_2e^{-c_2j^\al}
$$
where $c=\min(p_2/q_2,p_3/q_3)$ and $c_2$ and $C_2$ are some positive constants. The contribution to the right hand  side of (\ref{Bas}) coming from $\cD_{j,k}$,  with $j\geq (a+1)k$ is therefore controlled by
\[
\sum_{j=0}^{n-1}\sum_{k=0}^{j}e^{-c_2j^\al}\leq\sum_{j=1}^{\infty}j e^{-c_2j^\al}<\infty.
\]
Now we will control the contribution coming from $\text{Cov}_{\mu}(r, G_k\tilde G_j)$, when $j\geq(a+1)k$. First, using \eqref{p ass} and that $\|r\|_3<\infty$ we have
$$|\mu(rG_k)|\leq \|r\|_2\|G_k\|_2\leq C_0\|r\|_2 e^{c_0 k^\be}\leq C_0\|r\|_2e^{c_0 j^\beta}.$$
Next, using (\ref{Dec}) with $n=j$ and $s=rG_k$, (\ref{ExtCond}) we get that 
\[
\left|\bbE_\mu[rG_k\tilde G_j]\right|\leq C_0\|r\|_2e^{c_0 j^\beta}|\bbE_\mu[\tilde G_j]|+C_1b e^{j^\al}\del^{j-ak}.
\]
Moreover, 
\[
\left|\bbE_\mu[G_k\tilde G_j]\right|\leq b e^{j^\al}\del^{j-ak}
\]
where we have also used that $G_k$ is centered.
Next, using the Markov inequality we have that 
\[
|\bbE_\mu[\tilde G_j]|=|\bbE[G_j\bbI(|G_j|>e^{j^\al})]|\leq \|G_j\|_{p}^{1+q/p}e^{-(q/p) j^\al}
\]
where $q$ is the conjugate exponent of $p$ (and $p$ comes from \eqref{p ass}).
We conclude that, in absolute value, the contribution to the right hand side of (\ref{Bas}) coming from the pairs $j$ and $k$ so that $j\geq (a+1)k$ does not exceed a constant times
\[
\sum_{j=0}^{n-1}e^{j^\al}\sum_{k=0}^{j/(a+1)}\del^{j-ak}+\sum_{j=0}^{n-1}(j+1)e^{-c_2j^\al}\leq 
c\sum_{j=0}^{\infty}e^{j^\al}\del^{j/(a+1)}+\sum_{j=0}^\infty(j+1)e^{-c_2j^\al}<\infty
\]
where $c$ is some constant.
Now we estimate the contribution coming from pairs $(k,j)$ such that $k\leq j\leq (a+1)k$. First, by the Markov and the H\"older inequalities and (\ref{p ass}) we have
\[
\left|\text{Cov}_{\mu}(r,G_kG_j)\right|\leq \left|\text{Cov}_{\mu}(r,G_kG_j\bbI(|G_k G_j|\leq e^{j^\al}))\right|+ C_4e^{-c_4 j^\al}
\] 
where  $C_4$ and $c_4$ are some positive constants. Using now (\ref{Dec}) with $s=r$ and $f$ so that $f\circ\overline{ X}_k=G_jG_k\bbI(|G_j G_j|\leq e^{j^\al})$ we get that 
\[
\left|\text{Cov}_{\mu}(r,G_kG_j)\right|\leq\|r\|\del^{k}e^{j^\al}+C_4e^{-c_4 j^\al}.
\]
Therefore, there are constants $C_5,C_6>0$ so that
\[
\sum_{k=0}^{n-1}\sum_{j=k}^{(a+1)k}\left|\text{Cov}_{\mu}(r, G_kG_j)\right|\leq 
C_5\sum_{k=0}^{\infty}(k+1)e^{(a+1)^\al k^\al}\del^k+C_6\sum_{k=0}^{\infty}(k+1)e^{-c_4 k^\al}<\infty.
\]
\end{proof}

\begin{remark}
The arguments in the proof of Proposition \ref{Vprop} show that the conclusion of the proposition holds true if we assume that $\|G_k\|_p\leq \del^{-u k}$ for some sufficiently small $u$ and all $k\geq0$.
\end{remark}
\subsection{Quantitative versions for scalar-valued functions}
Let $X$ and $Y$ be random variables. Recall that the Kolmogorov (uniform) metric $d_K(X,Y)$ between the laws of $X$ and $Y$ is given by 
\[
d_K(X,Y)=\sup_{t\in\bbR}|P(X\leq t)-P(Y\leq t)|.
\]
We have the following (well known) version of the, so called, Berry-Esseen inequality:
\begin{lemma}\label{BE lemma}
Let $X$ and $Y$ be two real-valued random variables, and let $\varphi_X$ and $\varphi_Y$ be their characteristic functions, respectively. Let $Z$ be another random variable which has a bounded density function $f_Z$.
Then for every $T>0$ we have
\[
d_K(X,Y)\leq 4cd_K(Y,Z)+\int_{-T}^T\left|\frac{\varphi_X(t)-\varphi_Y(t)}{t}\right|dt+\frac{2\|f_Z\|_\infty c^2}{T}
\]
where $\|f_Z\|_\infty=\sup f_Z$ and 
$c>0$ is some absolute constant which can be taken to be the root of the equation
\[
\int_0^{c/2}\frac{\sin^2 x}{x^2}=\frac{\pi}4+\frac1 8.
\]
In particular,
\[
d_K(X,Z)\leq (4c+1)d_K(Y,Z)+\int_{-T}^T\left|\frac{\varphi_X(t)-\varphi_Y(t)}{t}\right|dt+\frac{2\|f_Z\|_\infty c^2}{T}.
\]
\end{lemma}
\begin{proof}
Taking $b=1$ at the beginning of Section 4.1 in \cite{Lin} we get that 
\begin{eqnarray*}
d_K(X,Y)\leq \int_{-T}^T\left|\frac{\varphi_X(t)-\varphi_Y(t)}{t}\right|dt\\+2T\sup_{x\in\bbR}\int_{-c/T}^{c/T}|P(Y\leq x+y)-P(Y\leq x)|dy.
\end{eqnarray*}
Next, it clear that for every $x\in\bbR$,
\begin{eqnarray*}
\int_{-c/T}^{c/T}|P(Y\leq x+y)-P(Y\leq x)|dy\leq 
\int_{-c/T}^{c/T}|P(Z\leq x+y)-P(Z\leq x)|dy\\+2cd_K(Y,Z)/T\leq 
\|f_Z\|_\infty\int_{-c/T}^{c/T}|y|dy+2cd_K(Y,Z)/T.
\end{eqnarray*}
\end{proof}

Lemma \ref{BE lemma} makes it possible to estimate $d_K(S_{\nu,n}/b_n,Z)$ by means of $d_K(S_{\mu,n}/b_n,Z)$, namely we can estimate the error in the weak convergence of $S_{\nu,n}/b_n$ by means of the error term in the weak convergence of $S_{\mu,n}/b_n$. 
\begin{theorem}\label{QuantEag}
Let Assumption \ref{Ass1} hold,  and suppose that $r\in B$ and that  the functions $g_j$ are real-valued.  Then for every positive integer $\rho<n$,  $T\geq1$ and a random variable $Z$ with a bounded density function $f_Z$ we have
\begin{eqnarray}\label{GenB}
d_K(S_{\nu,n}/b_n,Z)\leq (4c+1)d_K(S_{\mu,n}/b_n,Z)+\\ \frac{2T\mu\big(|S_\rho|(1+r)\big)}{b_n}+4\del_\rho\|r\|\ln T+\frac{2\|f_Z\|_\infty c^2}{T}+\frac{2\mu\left((r+1)|S_n|\right)}{b_n T}\nonumber
\end{eqnarray}
where $c$ comes from Lemma \ref{BE lemma}.
\end{theorem}

\begin{proof}
Let $t\not=0$ and set $t_n=t/b_n$.
As in the proof of Theorem \ref{Theorem1} for every $\rho<n$ we have
\[
|\mu(re^{it_n S_n})-\mu(e^{it_n S_n})|\leq |t|b_n^{-1}I_1(\rho)+|\mu(re^{it_n(S_n-S_\rho)})-\mu(e^{it_n (S_n-S_\rho)})|
\]
where 
\[
I_1(\rho)=\int |S_\rho(x)|(1+r(x))d\mu(x).
\]
When $|t|\leq1/T$ we will not use the above, and instead we will use the estimate
\[
|\mu(re^{it_n S_n})-\mu(e^{it_n S_n})|\leq \left|\big(\mu(re^{it_n S_n})-1\big)-\big(\mu(e^{it_n S_n})-1\big)\right|\leq \mu\left((r+1)|S_n|\right)|t_n|.
\]
Therefore, with $Y=S_{n,\mu}/b_n$ and $X=S_{n,\nu}/b_n$,
for every $\rho,T\geq 1$,
\begin{eqnarray*}
\int_{-T}^T\left|\frac{\varphi_X(t)-\varphi_Y(t)}{t}\right|dt\leq \int_{-1/T}^{1/T}\left|\frac{\varphi_X(t)-\varphi_Y(t)}{t}\right|+2TI_1(\rho)/b_n+\\
\int_{1/T\leq |t|\leq T}\left|\frac{\mu(re^{it_n(S_n-S_\rho)})-\mu(e^{it_n (S_n-S_\rho)})}{t}\right|dt\leq 2\mu\left((r+1)|S_n|\right)(b_n T)^{-1}\\+2TI_1(\rho)/b_n+4\del_\rho\|r\|\ln T
\end{eqnarray*}
where in the last inequality we have used Assumption \ref{Ass1}. The theorem follows now from Lemma \ref{BE lemma} and the above estimate.
\end{proof}

We remark that 
\[
\mu\left((r+1)|S_n|\right)\leq \|r+1\|_2\|S_n\|_{L^2}
\]
and so when $S_n$ has zero $\mu$-mean and $\|r\|_{L^2}<\infty$ we get that the above expression is of order $\sig_n=\sqrt{\text{Var}(S_n)}$. Hence, the contribution of the last expression in the right hand side of (\ref{GenB}) is of order $1/T$ when $b_n\thickapprox\sig_n$, which is the case in most  applications we have in mind. When, in addition, $r$ is bounded, $g_j$ are  uniformly bounded and $\del_j\leq c_1e^{-c_2j}$ for some positive $c_1$ and $c_2$ then by taking $T\leq An$ (for some $A>0$) and $\rho=c\ln b_n$ for a sufficiently large $c$  we get 
\[
d_K(S_{\nu,n}/b_n,Z)\leq (4c+1)d_K(S_{\mu,n}/b_n,Z)+C(b_n^{-1}T\ln b_n+1/T).
\]
Taking $T=b_n^{\frac12}$ we get that 
\[
d_K(S_{\nu,n}/b_n,Z)\leq (4c+1)d_K(S_{\mu,n}/b_n,Z)+Cb_n^{-1/2}\ln b_n.
\]
Since $T$ and its reciprocal appear in the right hand side of (\ref{GenB}) we do not expect to get better rates than the above only  under Assumption \ref{Ass1} (of course, certain rates can be obtained when $\del_j$ diverges polynomially fast to $0$ and when $r$ and $g_j$ only satisfy  certain moment conditions).

\begin{remark}
When $S_{n,\nu}$ is not centered (but $S_{n,\nu}$ is), then it is desirable to get estimates on $d_K(b_n^{-1}\overline{S_{n,\nu}},Z)$, where $\bar{Y}=Y-\bbE[Y]$ for every random variable $Y$. Applying Lemma 3.3 in \cite{HK-BE} with $a=\infty$ yields that 
\[
d_K(b_n^{-1}\overline{S_{n,\nu}},Z)\leq 3d_K(b_n^{-1}\overline{S_{n,\mu}},Z)+(1+4\|f_Z\|_\infty)|\bbE[S_{n,\nu}]-\bbE[S_{n,\mu}]|/b_n.
\]
The first expression on the above right hand side was estimated in Theorem \ref{QuantEag}, while the second expression was estimated in Section \ref{SecCent}. 
Using the above Lemma 3.3 together with  Proposition \ref{Vprop}  we can also  get rates in  the CLT for $(S_{n,\nu}-\bbE[S_{n,\nu}])/b_{n,\nu}$ from given rates in the corresponding CLT for $(S_{n,\mu}-\bbE[S_{n,\mu}])/b_{n,\mu}$.
\end{remark}

\subsubsection{Optimal convergence rates}
Consider the case when $b_n=n^{-\frac12}$ (or $b_n\thickapprox n^{\frac12}$) and 
\[
d_K(S_{n,\mu}/b_n,Z)=\mathcal O(n^{-1/2})
\]
where $Z$ is a standard normal random variable . The rate $n^{-1/2}$ is optimal, while Theorem \ref{QuantEag} is not likely to yield optimal rates for $d_K(n^{-\frac12}S_{n,\nu},Z)$ even when $\del_j$ decays exponentially fast to $0$ as $j\to\infty$ and $r$ and $\sup_j|g_j|$ are bounded (in this case we have managed to obtain the rate $n^{-1/4}\ln n$).
In many situations (see \cite{DFGTV1}, \cite{DFGTV2}, \cite{HafSDS} and \cite{HK})
there exist normed spaces $B_n$ of functions on some  measurable spaces $\cE_n$, a family of operators $\cL^{(n)}_{z}:B\to B_n, z\in\bbC$ and a family of probability measures $\mu_n$ on $\cE_n$ so that for every $s\in B$, $n\geq1$ and $z\in\bbC$,
\[
\mu(s\cdot e^{z S_n})=\mu_n(\cL^{(n)}_{z}s).
\] 
Moreover, the norm on $B_n$ is larger than the $L^1(\mu_n)$-norm and there exists $\epsilon>0$, $C>0$ and a function $R:[0,\infty)\to\bbR$ so that $R(t^2)$ is integrable
and  for every $t\in[-\epsilon,\epsilon]$, $n\geq1$ and a function $s$ with $\int sd\mu=0$,
\[
\|\cL_{it}^{(n)}s\|\leq C\|s\||t|R(tn^2).
\]
In applications such estimates follow from analyticity (in $z$) assumptions on the operators $\cL_z^{(n)}$ together with a complex sequential Ruelle-Perron-Frobeneius theorem (see \cite[Theorem 3.3]{HafSDS} for example).
 When $r-1\in B$ then 
with $t_n=tn^{-1/2}$,
\[
|\mu(re^{it_n S_n})-\mu(e^{it_n S_n})|=|\mu_n(\cL_{it_n}^{(n)}(r-1))|\leq C_1|t|n^{-\frac12}R(t^2)
\]
where we have used that $\mu(r-1)=0$.
Taking $T\approx\del\sqrt n$ in Lemma \ref{BE lemma} we get that
\[
d_K(S_{\nu,n}/b_n,Z)\leq C(4c+1)d_K(S_{\mu,n}/b_n,Z)+C\left(\frac{2\|f_Z\|_\infty c^2}{\del}+C_1\int R(t^2)dt\right)n^{-1/2}
\]
and so the optimal rate of convergence is preserved in the above circumstances. 
When $|\bbE[S_{n,\nu}]-\bbE[S_{n,\mu}]|$ is bounded in $n$ (see Section \ref{SecCent}) we also obtain optimal convergence rates in the convergence of $(S_{n,\nu}-\bbE[S_{n,\nu}])/b_n$ from the corresponding optimal rate for $(S_{n,\mu}-\bbE[S_{n,\mu}])/b_n$. Using Proposition \ref{Vprop}, if $b_n=b_{n,\mu}$ then we can replace $b_n$ with $b_{n,\nu}$ in the CLT corresponding to $\nu$ and still get the optimal rate.

\section{Infinite dimensional results}

\subsection{The weak invariance principle (WIP)}
Let $g_j:\cX_j\to\bbR^d$, $j\geq0$ be  vector-valued functions and for every $t\geq 0$ and $n\geq1$ consider the function
\[
\cS_n(t)=\sum_{n=0}^{[nt]-1}g_j\circ \overline{X}_j=\sum_{n=0}^{[nt]-1}g_j(\overline{X}_j(x)).
\]
For every fixed $n$, we can view $\cS_n(t)$ as a continuous time process  by considering  $x$ as a random variable whose distribution is either $\mu$ or $\nu=rd\mu$, where $r$ is a density function. 
Let $\cS_{n,\mu}(t)$ and $\cS_{n,\nu}(t)$ be the resulting continuous time processes.

\begin{theorem}\label{Theorem}
Suppose that 
for some two conjugate exponents $p$ and $q$ we have that 
 $r$ lies in the $L^p(\mu)$-closure of $B\cap L^p(\mu)$ and $g_j\circ X_j\in L^q(\mu)$ for all $j\geq0$.
 Let $(b_n)$ be a sequence of positive numbers so that $\lim_{n\to\infty}b_n=\infty$. Then, under Assumption \ref{Ass1},
if the sequences of processes process $\cS_{n,\mu}(\cdot)/b_n$ converges in distribution in the Skorokhod space $D([0,\infty),\bbR^d)$ as $n\to\infty$, then the processes $\cS_{n,\nu}/b_n$ converges in distribution to the same limit. If  $r$ is positive ($\mu$-almost surely) then the convergence with respect to $\mu$ can be derived from the convergence with respect to $\nu$.
\end{theorem}

\begin{proof}
Suppose that $\cS_{n,\mu}(\cdot)/b_n$ converges in distribution.
Then $\cS_{n,\mu}(\cdot)/b_n$  is a tight family, namely for every $\ve>0$ there exists a compact set $K_\ve$ (of paths in the Skorokhod space) so that 
\[
\sup_n\mu\{\cS_{n,\mu}/b_n(\cdot)\not\in K_\ve\}<\ve.
\]
We claim that $\cS_{n,\nu}/b_n$ is also a tight family. 
 Indeed, for any $C>0$
set 
\[
\eta_C=\mu(|r|I(|r|>C)).
\]
Since $r\in L^1(\mu)$ we have $\lim_{C\to\infty}\eta_C=0$.
For every positive $\ve$ and $C$ we have
\begin{eqnarray*}
\nu\{\cS_n(\cdot)\not\in K_{\ve}\}=\mu(r I(\cS_n(\cdot)\not\in K_{\ve}))\\\leq 
C\mu(I(\cS_n(\cdot)\not\in K_{\ve}))+\mu(rI(|r|>C)\leq \eta_C+C\ve.
\end{eqnarray*}
Given $\ve'>0$ we first take $C$ large enough so that $\eta_C<\frac12\ve'$, and then, after fixing this 
$C$, we take $\ve$ so that $C\ve<\frac12\ve'$. Then set $K_{\ve}$ satisfies
\[
\sup_n\nu\{\cS_n(\cdot)\not\in K_{\ve}\}<\ve'.
\]
Tightness and the convergence of all the finite dimensional distributions, expect  from the ones which involve members of a certain set of $t$'s which depend only on the target limiting distribution, is equivalent to weak convergence in the Skorokhod space (see \cite[Theorem 15.1]{Bil}). Therefore, what is left  to prove in order to get the convergence in distribution of $\cS_{n,\nu}/b_n$ is that one can derive the convergence of a given finite dimensional distribution of $\cS_{n,\nu}/b_n$  from the convergence of the corresponding  finite dimensional distribution of $\cS_{n,\mu}/b_n$ (in fact, we will show that these two convergences are equivalent). 
Consider the sequence
\[
c_k=\int |S_k(x)|d\mu(x)+\int r(x) |S_k(x)|d\mu(x)
\]
and let $(a_n)$ be a sequence of positive integers so that $c_{a_n}=o(b_n)$ and $\lim_{n\to\infty}a_n=\infty$. It is clear that we can assume that $a_n=o(n)$.
Let $s_1,...,s_m\in(0,\infty)$ and  $t\in\bbR^{dm}$. For every $n\geq1$ we write $t_n=t/b_n$.
Set 
\[
V_n(x)=(S_{[ns_1]}(x),...,S_{[ns_m]}(x))
\]
and
\[
U_n(x)=(S_{a_n}(x),...,S_{a_n}(x)).
\]
We first assume that $r\in B$.
By the mean value theorem, for all sufficiently large $n$ we have
\[
\left|\mu(re^{it_n V_n})-\mu(re^{it_n (V_n-U_n)})\right|\leq C_m|t|b_{n}^{-1}\mu(r\cdot |S_{a_n}|)\leq C_m|t|c_{a_n}/b_{n}\to 0\text { as }n\to\infty
\]
where $C_m$ is some constant which depend only on $m$ (the number of $s_i$'s).
Relying on Assumption \ref{Ass1}, taking into account that $\mu(r)=1$ and that $V_{n}-U_{a_n}$ is a function of $\overline{X}_{a_n}$,
we have
\[
|\mu(r\cdot e^{it_n(V_n-U_n)})-
\mu(e^{it_n(V_n-U_{n})})|\leq \|r\|\del_{a_n}\to 0\text { as }n\to\infty.
\] 
Finally, by the mean value theorem we have
\[
\left|\mu(e^{it_n(V_n-U_n)})-\mu(e^{it_n V_n})\right|\leq C_m|t|b_{n}^{-1}\mu(|S_{a_n}|)\leq C_m|t|c_{a_n}/b_n\to 0\text { as }n\to\infty.
\]
We conclude that for every $t\in\bbR^{md}$, 
\[
\lim_{n\to\infty}|\mu(re^{it_nV_n})-\nu(e^{it_nV_n})|=0
\]
and the claim about the equivalence between the convergence of the finite dimensional distributions follows from the Levi continuity theorem.
The reduction to the case when $r$ is only in the $L^p(\mu)$-closure of $B$ relies on the inequalities
\[
|\mu(re^{it_n V_n})-\mu(s e^{it_n V_n})|\leq\|r-s\|_{L^p}
\]
and 
\[
\int s(x)|S_n(x)|d\mu(x)\leq \int r(x)|S_n(x)|d\mu(x)+\|s-r\|_{L^p}\|S_n\|_{L^q}<\infty
\]
where $p$ and $q$ come from the assumptions of the theorem. 
\end{proof}

\begin{remark}
 When $|\bbE[S_{n,\nu}]-\bbE[S_{n,\mu}]|=o(b_n)$ (see Section \ref{SecCent} for conditions insuring that) we also obtain the convergence of $(S_{[nt],\nu}-\bbE[S_{[nt],\nu}])/b_n$ from the convergence of 
 $(S_{[nt],\mu}-\bbE[S_{[nt],\mu}])/b_n$. 
\end{remark}

\begin{remark}\label{WIP}
Theorem \ref{Theorem} shows that the weak invariance principles which follow from the results in  \cite{DFGTV0} and \cite{DH} hold true also when starting from the Lebesgue measure on the underlying manifold, and not only from the equivariant random measure $\mu_\om$.
\end{remark}

\subsection{The iterated weak invariance principle}
In this section we will discuss a version of Eagleson's theorem for the iterated weak invariance principle. Since the latter is less known  than the usual WIP, we will first describe the context in which it is naturally arises, 

Let $d,e\in\bbN$, $a:\bbR^d\to\bbR^d$ be a function of class $C^{1+}$  and $b:\bbR^{d}\to\bbR^{d+e}$ be a function of class $C^{2+}$.
Consider the stochastic  differential equation (SDE)
\begin{equation}\label{ODE}
dX=\left(a(X)+\frac{1}2\sum_{\al,\be,\gamma}D^{\be \gamma}\partial^{\al}b^\be(X)b^{\al\gamma}(X)\right)dt+b(X)\circ dW,\,\,\,X(0)=\xi
\end{equation}
where $W$ is a standard $d$-dimensional Brownian motion and $\xi$ is a fixed vector in $\bbR^d$. 
Here we sum over $1\leq\al\leq d,\,1\leq\be,\gamma\leq e$, and $b^{\al \gamma}$
and $b^\be$ denote the $(\al,\gamma)$-th entry and $\beta$-th column, respectively, of $b$.
In \cite[Theorem 2.2]{KM}, for suspension flows $\{\phi_s:\,s\geq 0\}$ built over  non-uniformly expanding or hyperbolic dynamical systems $(X,T,\mu)$, it was shown that for any centered function $v:X\to\bbR^e$, the solution $X_n$ of the SDE
\[
dX_n=a(X_n)dt+b(X_n)dW_n,\,\,\,X_n(0)=\xi
\]
weakly converges in $C([0,\infty),\bbR^d)$ towards the solution of  \eqref{ODE} as $n\to\infty$, 
where 
$$W_n(t)=n^{-1/2}\int_0^t v\circ\phi_s ds\,\,\text{ and }\,\,dW_n=\dot{W_n}dt.$$ 
 Next, let $\bbW_n(t)$ be the $e\times e$-dimensional process whose entries are given by 
 $$\bbW_n^{\be,\gamma}(t)=\int_0^t W_n^\be dW_n^\gamma,\,\,\, 1\leq \be,\gamma\leq e.$$
A key ingredient in the proof of \cite[Theorem 2.2]{KM} is the iterated WIP \cite[Theorem 2.1]{KM} which states that $(W_n,\bbW_n)$ weakly converges as $n\to\infty$ towards a process with a certain  structure. The proof of the latter was based on a discretization argument, where Kelly and Melbourne showed that it is enough to prove the convergence of $(S_n,\bbS_n)$ towards the latter process,  where $S_n(t)=n^{-1/2}\sum_{j=0}^{[nt]-1}v\circ T^j$ and 
$$
\bbS_n^{\be,\gamma}(t)=\int_0^t S_n^\be dS_n^\gamma=n^{-1}\sum_{0\leq i<j\leq [nt]-1}v^\be\circ T^i\cdot v^\gamma\circ T^j,\,\,\,1\leq \be,\gamma\leq e.
$$
In fact, the iterated WIP needed is with respect to the measure $\nu=rd\mu$, where $r$ is the underlying roof function defining the suspension flow, and not with respect to $\mu$. In order to settle this the authors of \cite{KM} used a version of Eagleson's theorem from \cite{ZW} which applies, in particular, to the iterated WIP.

A natural question arising here is whether the smooth approximation still holds  for random suspension flows (see, for instance \cite{Kif LimThms} for the definition), namely if we replace $T$ with a random dynamical system (as describe in Example \ref{Eg1}) and the roof function $r$ with a random roof function. It seems to us that the main obstacle in such a generalization is proving a version of Eagleson's theorem for the iterated WIP in the  random dynamics setup. In what follows we will provide such results in the more general non-stationary setup of this paper. 

Let $(\cE,\cF,\mu)$ and $X_0,X_1,...$ be as specified in Section \ref{Sec1}. 
Let $d, m\in \mathbb N$. For each $1\leq i\leq d$, let $g_j^{(i)}:\cY_j \to\bbR^{d_i},\,j\geq0$
be a sequence of vector-valued functions, where $d_i\geq1$. Set
\[
S_n g^{(i)}=\sum_{j=0}^{n-1}g^{(i)}_j\circ\overline X_j.
\]
Moreover, for each $d+1\leq k\leq d+m$ let $f^{k,u}_j$ and $f^{k,v}_j$, $j\geq0$ be sequences of real-valued functions on $\cY_j$, where $1\leq u\leq q_{k}$ and $1\leq v\leq p_k$ and $p_k, q_k \in \mathbb N$. Consider the matrix $A_n f_k$ given by
\[
(A_n f_k)_{u,v}=\sum_{i=0}^{n-1}\sum_{j=i+1}^{n-1}f^{k,u}_{i}\circ \overline{X}_i\cdot f^{k,v}_j\circ\overline{X}_j=
\sum_{i=0}^{n-1}f^{k,u}_{i}\circ\overline{X}_i \cdot \big(S_nf^{k,v}-S_{i+1}f^{k,v}\big),
\]
We consider $A_n f_k$ as an $\bbR^{p_k\cdot q_k}$-valued  function. 
For each  $t\in \mathbb R$ consider the vector-valued function  $V_n(t)$ on $\cE$  given by $V_n(t)=(L_n(t), R_n(t))$, where
\[
L_n(t)=n^{-\frac12}(S_{[nt]} g^{(1)},...,S_{[nt]} g^{(d)})
\]
and 
\[
R_n(t)=n^{-1}(A_{[nt]} f_{d+1},...,A_{[nt]} f_{d+m}).
\]

Let $r$ be a function in the $L^{s_1}(\mu)$-closure of $B\cap L^{s_1}(\mu)$, for some $s_1\geq1$, where $B$ comes from Assumption \ref{Ass1}.
Assume also that $r\geq0$ and that $\int r \, d\mu=1$.
Let us introduce an additional (moment) assumption. 
\begin{assumption}\label{Assumption2}
There are $p_1\geq1$ so that $s_1\geq p_1^*=p_1/(p_1-1)$ and $\ve\in(0,1)$ such that for every $d+1\leq k\leq d+m$ and $1\leq v\leq q_k$
we have
\begin{equation}\label{tz}
\|S_n f^{k,v}\|_{L^{p_1}(\mu_\om)}\leq Cn^{1-\ve},
\end{equation}
where $C$ is some constant. Moreover,  
\[
f_{i}^{k,u}\circ\overline{X}_i\in L^{s_3}(\mu)
\]
for every $i\geq0$, $d+1\leq k\leq d+m$ and $1\leq u\leq p_k$, where $s_3$ is some real number (where $1^*:=\infty$) and 
\[
\frac{1}{s_3}=1-\frac 1{s_1}-\frac 1{p_1}.
\]
Here, we use the conventions $\frac 1{\infty}:=0$ and $\frac 1{0}:=\infty$.
\end{assumption}
The condition \eqref{tz} holds true with $\ve=1/2$ in the random dynamics setup for appropriate classes of random non-uniformly expanding or hyperbolic maps (or for Markov chains in random dynamical environments \cite[Ch. 6]{HK}).

Let $\nu$ be the probability measure on $\cE$ given by $d\nu=r d\mu$. 
Our main result here is the following theorem.
\begin{theorem}\label{Theorem2}
Under the Assumptions \ref{Ass1} and \ref{Assumption2}, if the continuous time process $V_n(\cdot)$ converges in distribution in the Skorokhod  space with respect to the measure $\mu$, then it also converges in distribution in the Skorokhod  space with respect to the measure $\nu$ (and to the same limit). Finally, if
$r>0$ ($\mu$-a.s.),  then the convergence with respect to $\mu$ can be derived from the convergence with respect to $\nu$.
\end{theorem}

\subsection*{Proof of Theorem \ref{Theorem2}}
First  if $V_n(\cdot)$ converges in distribution with respect to $\mu$, then $\{V_n\}$ is a tight family. Arguing exactly as in the proof of Theorem \ref{Theorem}, we obtain that it is also a tight family with respect to the measure $\nu$. Therefore, it remains to show that the finite dimensional distributions converge.

Let $t_1,...,t_{p}$ be positive real numbers and set
\[
Q_n=(V_n(t_1),V_n(t_2),...,V_n(t_{p})).
\]
We first need the following elementary result.
\begin{lemma}\label{Elem}
For any two sequences $(c_n)$ and $(q_n)$ of real numbers such that $\lim_{n\to\infty}q_n=\infty$, there exists a (weakly increasing) sequence $(b_n)$ of natural numbers which tends to $\infty$ as $n\to\infty$ so that for any other sequence $(a_n)$ of natural numbers which tends to $\infty$ and satisfies $a_n\leq b_n$ we have $c_{a_n}=o(q_n)$.
\end{lemma}
\begin{remark}\label{Rema}
We now observe that for two pairs of sequences $(c_n)$ and $(q_n)$ as in Lemma~\ref{Elem}, we can choose a sequence $(b_n)$ compatible with both of these pairs. More precisely, 
take  two pairs of sequences $(c_n), (q_n)$ and $(c'_n),(q'_n)$ such that $\lim_{n\to \infty} q_n=\lim_{n\to \infty} q_n'=\infty$.  Let $(b_n)$ be a sequence given by Lemma~\ref{Elem} for the pair $(c_n), (q_n)$. Furthermore, let $(b_n')$ be a sequence given by Lemma~\ref{Elem} for the pair $(c'_n),(q'_n)$. 
Set $b_n'':=\min \{b_n, b_n' \}$, $n\in \mathbb N$. Then, $(b_n'')$ is weakly increasing and for any sequence of natural numbers $(a_n)$ such that $a_n\le b_n''$ and $\lim_{n\to \infty}a_n=\infty$, we have that 
 $c_{a_n}=o(q_n)$ and $c'_{a_n}=o(q_n')$.
\end{remark}
\begin{corollary}\label{Cor}
Under Assumptions \ref{Ass1} and \ref{Assumption2}, there exists a sequence $(b_n)$ which tends to $\infty$ as $n\to\infty$ such  that for any other sequence $(a_n)$ of natural numbers such that $a_n\leq b_n$ for all $n$ and $\lim_{n\to \infty}a_n=\infty$, for any relevant  $i$, $k$, $u$ and $v$ we have  that 
\[
\lim_{n\to\infty}\del_{1,i}(n)=\lim_{n\to\infty}\del_{2,i}(n)=\lim_{n\to\infty}\del_{1,k,u,v}(n)=\lim_{n\to\infty}\del_{2,k,u,v}(n)=0
\]
where
\begin{eqnarray*}
\del_{1,i}(n)=n^{-1/2}\mu(|S_{a_n} g^{(i)}|),\,\,\del_{2,i}(n)=n^{-1/2}\mu(|r\cdot S_{a_n} g^{(i)}|),\\\del_{1,k,u,v}(n)=n^{-1}\mu\Big(\big|\sum_{i=0}^{a_n-1}f^{k,u}_{i}\circ \overline{X}_i\cdot \big(S_nf^{k,v}-S_{i+1}f^{k,v}\big)\big|\Big),\\\text{and }\,\,\,
\del_{2,k,u,v}(n)=n^{-1}\mu\Big(\big|r\cdot\sum_{i=0}^{a_n-1}f^{k,u}_{i}\circ \overline{X}_i\cdot \big(S_nf^{k,v}-S_{i+1}f^{k,v}\big)\big|\Big).
\end{eqnarray*}
\end{corollary}

\begin{proof}
First, by ~\eqref{tz} for all $0\leq i<n$  we have that we have,
\[
\|S_nf^{k,v}-S_i f^{k,v}\|_{L^{p_1}(\mu)}\leq 2Cn^{1-\ve}.
\]
Standard applications of the H\"older inequality yield that for $l=1,2$, 
\[
\del_{l,k,u,v}(n) \leq n^{-1}\sum_{i=0}^{a_n-1}\|f_{i}^{k,u}\circ\overline{X}_i\|_{L^{s_3}(\mu)}2C_lCn^{1-\ve} 
=2CC_ln^{-\ve}c_{k,u,a_n}
\]
where $C_1=1$, $C_2=\|r\|_{L^{s_1}(\mu)}$ and
\[
c_{k,u,n}=\sum_{i=0}^{n-1}\|f_{i}^{k,u}\circ\overline{X}_i\|_{L^{s_3}(\mu)}.
\]
Here $s_1$ and $s_3$ come from Assumption \ref{Assumption2}.
By applying Lemma \ref{Elem} with 
\[
c_n=\max_i\left(\mu(|S_{n} g^{(i)}|)+\mu(|r S_{n} g^{(i)}|)\right)\,\,\text{ and }\,\,q_n=n^{1/2} 
\]
and then with 
\[
c_n=\max_{k,u}c_{k,u,n},\,\text{ and }\,\,q_n=n^{\ve} 
\]
we complete the proof of the  corollary, taking into account Remark \ref{Rema}.
\end{proof}

In order to complete the proof of Theorem~\ref{Theorem2} we need the following result.

\begin{proposition}\label{Prop}
Suppose that Assumptions \ref{Ass1} and \ref{Assumption2} hold true.
Then the sequence of random vector-valued variables $Q_n$ converges in distribution with respect to $\mu$ if and only if it converges in distribution with respect to $\nu$ (and in the latter case the limiting distributions are equal). 
\end{proposition}

\begin{proof}
Let $D$ denote the dimension of the range of the (random) functions $Q_n$. By the Levi continuity theorem, in order to prove the proposition it is enough to show that for all $s\in\bbR^D$ we have
\begin{equation}\label{Target}
\lim_{n\to\infty}|\mu(r e^{isQ_n})-\mu(e^{isQ_n})|=0.
\end{equation}
Since we can approximate $r$ in $L^{s_1}(\mu)$ by non-negative functions $s\in B\cap L^{s_1}(\mu)$ satisfying $\mu(s)=1$ and since
\[
|\mu(r e^{isQ_n})-\mu(s e^{isQ_n})|\leq\|r-s\|_{L^{s_1}(\mu)},
\]
we conclude that  is enough to prove (\ref{Target}) in the case when $r\in B\cap L^{s_1}(\mu)$. Note that Assumption \ref{Assumption2} is left unchanged after such a reduction.

Let $(a_n)$ be a sequence such that the conclusion of Corollary \ref{Cor} holds. It is clear that we can assume without loss of generality that $a_n<n$.
For  $1\leq i\leq d$ and $1\leq \rho<n$, set
\[
G_{n,\rho} g^{(i)}=S_ng^{(i)}-S_{\rho} g^{(i)}
\]
which is a function of the variable $\overline{X}_{\rho}$.
Then  for all $1\leq\ell\leq p$ we have
 \[
S_{[n t_\ell]} g^{(i)}=S_{a_{[nt_\ell]}} g^{(i)}+G_{[nt_\ell],a_{[nt_\ell]}} g^{(i)}.
 \]
Next, for every $d+1\leq k\leq m$ and  $1\leq \rho <n$ set 
\[
(G_{n,\rho} f_k)_{u,v}=\sum_{i=\rho}^{n-1}f^{k,u}_{i}\circ \overline{X}_i\cdot \big(S_nf^{k,v}-S_{i+1}f^{k,v}\big).
\]
Then
$G_{n,\rho} f_k$ is a function of $\overline{X}_\rho$ and 
for every $1\leq\ell\leq p$ we have 
\[
A_{[nt_\ell]} f_k=\sum_{i=0}^{a_{[nt_\ell]}-1}f^{k,u}_{i}\circ\overline{X}_i \cdot\big(S_{[nt_\ell]}f^{k,v}-S_{i+1}f^{k,v}\big)
+G_{[nt_\ell],a_{[nt_\ell]}}f_k
\] 
Let  $\tilde Q_n$ be defined by similarly to $Q_n$ but with $G_{[nt_\ell],a_{[nt_\ell]}} g^{(i)}$ and $G_{[nt_\ell],a_{[nt_\ell]}} f_k$ instead of $S_{[n t_\ell]} g^{(i)}$
and $A_{[nt_\ell]} f_k$, respectively.
Using the above identities, by the mean value theorem we have
\begin{eqnarray*}
\left|\mu(r e^{is Q_n })-\mu(r e^{is\tilde Q_n})\right|\leq C|s|\max_{1\leq \ell\leq p}\Bigg(n^{-1/2}\sum_{i=1}^{d}\mu(\big|r\cdot S_{a_{[nt_\ell]}} g^{(i)}\big|)\\+n^{-1}\sum_{k=d+1}^{d+m}\sum_{u,v}\mu\Big(\Big|r\cdot \sum_{i=0}^{a_{[nt_\ell]}-1}f^{k,u}_{i}\circ \overline{X}_i\cdot \big(S_{[nt_\ell]}f^{k,v}-S_{i+1}f^{k,v}\big)\Big|\Big)\Bigg)
\to 0
\end{eqnarray*}
as $n\to \infty$, 
where $C$ is some constant which depends only on the dimension $D$. The same argument gives that 
\begin{eqnarray*}
\left|\mu(e^{is \tilde Q_n })-\mu(e^{is Q_n})\right|\leq C|s|\max_{1\leq\ell\leq p}\Bigg(n^{-1/2}\sum_{i=1}^{d}\mu\Big(|S_{a_{[nt_\ell]}} g^{(i)}|)\\+n^{-1}\sum_{k=d+1}^{d+m}\sum_{u,v}\mu\Big(\Big|\sum_{i=0}^{a_{[nt_\ell]}-1}f^{k,u}_{i}\circ \overline{X}_i \cdot \big(S_{[nt_\ell]}f^{k,v}-S_{i+1}f^{k,v}\big)\Big|\Big)\Bigg)
\to 0
\end{eqnarray*}
as $n\to \infty$. 
Finally, using Assumption \ref{Ass1}, taking into account that $\mu(r)=1$ 
and that the function $H=e^{is\tilde Q_n}$ is a bounded function of the variable $\overline{X}_{a_{p_n}}$, where $p_n=p_n(t)=\min\{[nt_1],...,[nt_p]\}$,
we have
\[
|\mu(r e^{is \tilde Q_n})-
\mu(e^{is \tilde Q_n})|= O(\del_{a_{p_n}}),
\] 
and thus, as $\lim_{n\to\infty}a_{p_n}=\infty$,
\[
|\mu(e^{is \tilde Q_n}r)-
\mu(e^{is \tilde Q_n})| \to 0, 
\]
when $n\to \infty$.  It remains to observe that~\eqref{Target} follows from the last three assertions. 
\end{proof}


\begin{acknowledgement*}
I would like to thank Davor  Dragi\v cevi\' c for suggesting me to write this paper and for many useful conversations. I would also like to thank Paul Doukhan for several references to examples of non-stationary and mixing stochastic processes. 
\end{acknowledgement*}

\end{document}